\documentclass[12pt]{article}

\usepackage{indentfirst}
\usepackage{footmisc}

\usepackage{amsthm}

\usepackage[latin1]{inputenc}
\usepackage{amsmath,amssymb}
\usepackage{amsmath,accents}

\usepackage{alltt,graphics,graphpap,color}
\usepackage[dvips]{graphicx}

\usepackage{amsfonts}
\usepackage{amscd}

\usepackage{mathrsfs}
\usepackage{hyperref}
\usepackage{ marvosym }

\hypersetup{linktoc=page,colorlinks=true,linkcolor=red,  citecolor=blue, }

\newtheorem{Theorem}{Theorem}[section]

\newtheorem{Definition}[Theorem]{Definition}

\newtheorem{Lemma}[Theorem]{Lemma}

\newtheorem{Remark}[Theorem]{Remark}

\newtheorem*{Theorem*}{Theorem}

\newcommand{\cvd}{\hfill$\square$ \bigskip}

\newcommand{\cc}{\mathcal {C}}
\newcommand{\hl}{\mathcal {L}}

\newcommand{\hh}{\mathbb H}
\newcommand{\rr}{\mathbb R}

\begin{document}

\title{Constant mean curvature hypersurfaces in $\hh^n\times\rr$ with small planar boundary}

\author{\sc Barbara Nelli \footnote{Barbara Nelli, Department of Information Engineering, Computer Science and Mathematics, Universit\`a degli Studi dell'Aquila, via Vetoio 1, 67100 L'Aquila, Italy. E-mail: barbara.nelli@univaq.it} \& Giuseppe Pipoli \footnote{Giuseppe Pipoli, Department of Information Engineering, Computer Science and Mathematics, Universit\`a degli Studi dell'Aquila, via Vetoio 1, 67100 L'Aquila, Italy. E-mail: giuseppe.pipoli@univaq.it}}
\date{}

\maketitle


{\small \noindent {\bf Abstract:}
We show that  constant mean curvature hypersurfaces  in ${\mathbb H}^n\times\rr$, with small and pinched boundary  contained in a horizontal slice  $P$  are topological disks, provided  they are contained in one of the two halfspaces determined by $P.$
This is the analogous in ${\mathbb H}^n\times\rr$ of a result in $\rr^3$ by  A. Ros and H. Rosenberg \cite[Theorem 2]{RR}.

}
 
\medskip

 \let\thefootnote\relax\footnote{The authors were partially supported by INdAM-GNSAGA.}



\bigskip

\section{Introduction}

There is little known about  the topological and geometrical structure of constant mean curvature hypersurfaces  with convex  boundary. 

For example, it is unknown if  a surface embedded in $\rr^3,$  with boundary a circle and constant mean curvature,  is isometric to a spherical cap.
During the years, there have been partial results concerning this problem.  Let us recall the ones that we consider  the deepest.

The result by Brito, W. Meeks, H. Rosenberg  and R. Sa Earp \cite{BMRS}, yields  that an embedded constant mean curvature surface with boundary a circle in a plane $P$, is a spherical cap, provided it is transverse to $P$ along the boundary. 
In fact the authors are able to prove that  the transversality condition forces the surface to  stay in one 
of the two halfspaces determined by $P.$ Then one can use Alexandrov reflection method to get the result.

Notice that, A. D. Alexandrov states that a closed, embedded   surface  with constant mean curvature in $\rr^3$ is a round sphere  \cite{H} (see also \cite{N} for a survey about the subject). 
In light of Alexandrov's theorem, it is reasonable to expect that, if the boundary curve of a constant mean curvature surface  $M$ is small, then $M$ is a topological disk. By rescaling, this is analogous to expect that $M$ is topologically a disk if the mean curvature of $M$ is small when compared with the curvature of it's boundary.

 Indeed, A. Ros and H. Rosenberg \cite{RR}   showed that if $\Gamma$ is a  convex curve contained in the plane $P$,  the mean curvature $H$ of an embedded constant mean curvature surface $M$ is small when compared with the curvature of $\Gamma$ and if $M$ is contained in the halfspace bounded by $P$, then $M$ is a topological disk. 
Their result was extended in the hyperbolic $3$-space $\mathbb{H}^3$ by B. Semmler \cite{S}, and in $\rr^n,$ for all symmetric functions of the principal curvatures,  by B. Nelli and B. Semmler \cite{NS}.

In this paper we extend Ros-Rosenberg result to constant mean curvature  hypersurfaces in ${\mathbb H}^n\times\rr$.
Namely, we prove the following theorem (see Theorem \ref{theorem-main}). The notion of $r_{ext}$ (exterior radius) and $r_{int}$ (interior radius) in the statement, are intuitively clear. For a precise definition, as well as for the definition of horoconvexity,  see the beginning of Section \ref{proof}.

\begin{Theorem*}
Let $M$  be a $n$-dimensional hypersurface of $\hh^n\times\rr$ embedded in  $\hh^n\times]0,\infty]$ with constant mean curvature $H>\frac{n-1}{n}$ and boundary $\partial M=\Gamma$. Assume that $\Gamma$ is a closed $(n-1)$-dimensional horoconvex hypersurface of the slice  $P=\hh^n\times\{0\}$ satisfying the pinching $2r_{int}>r_{ext}$. Then there is a constant $\delta(n,H)>0$ depending only on $n$ and $H$ such that if $r_{ext}\leq\delta(n,H)$, then $M$ is topologically a disk.
Moreover either  $ M$ is a graph over the domain $\Omega$ bounded by $\Gamma$ or $N= M\cap(\Omega\times]0,\infty[)$ is a graph over $\Omega$ and $ M\setminus N$ is a graph over $\Gamma\times]0,\infty[$ with respect to the lines orthogonal to $\Gamma\times]0,\infty[.$
\end{Theorem*}

We emphasize that the assumption $H>\frac{n-1}{n}$ in the previous Theorem  is not restrictive because the following theorem holds.

\begin{Theorem*}
Let $\Omega$  be a connected domain in a horizontal section $P$ of  ${\mathbb H}^n\times{\mathbb R}$, with horoconvex boundary $\Gamma=\partial\Omega$. If $M$ is an embedded constant mean curvature  hypersurface with boundary $\Gamma$ and mean curvature $H\leq \frac{n-1}{n}$, then $M$ is a graph on $\Omega$. 
In particular $M$ is a topological disk.
\end{Theorem*}

The previous Theorem  for $n=2$  is proved  by B. Nelli, R. Sa Earp, W. Santos and E. Toubiana in \cite[Theorem 2.2]{NSST} and  by P. Berard and R.  Sa Earp in \cite[Theorem 3.3]{BS} for $n>2$.

Ros and  Rosenberg's proof relies on a crucial rescaling theorem (see \cite[Theorem 1]{RR}).
In trying to adapt their proof to $\hh^n\times\rr$, however, one is faced with the first obstacle that rescalings are isometries in $\hh^n$. 
This obstacle was overcome by B. Semmler in \cite{S}, where she uses constant mean curvature  horizontal half-cylinders as barriers to give a different proof of \cite[Theorem 2]{RR}.   In order to be able to use Semmler's type proof we need   {\em horizontal cylinders}, {\em i.e.} constant mean curvature hypersurfaces invariant by hyperbolic translations in $\hh^n\times\rr.$ Such cylinders for $n=2$ are described  by M. Manzano and I. Onnis  \cite{Man, On} and  by P. Berard and R. Sa Earp \cite{BS} for $n>2.$ 
Then, one is faced with the  second obstacle that, when $H\longrightarrow \frac{n-1}{n},$ compact  constant mean curvature, invariant by rotations hypersurfaces in  $\hh^n\times\rr,$  converge to a  complete non compact graph different from  a slice (see Remark \ref{no-flat}). In order to overcome this second obstacle we need two ingredients:
\begin{itemize}
\item a precise estimate of  how much we can go beyond the boundary of $M$, when doing  Alexandrov reflection: Lima\c con construction, see Section \ref{sez-lima};
\item the use of immersed constant mean curvature hypersurfaces in $\hh^n\times\rr$ invariant by rotation and horizontal cylinders, as barriers. 
\end{itemize} 

The paper is organized as follows.
In Section \ref{sez-lima}, we describe the construction of the hyperbolic  Lima\c con and study  its geometry. 
In Section \ref{sez-barriere}, we describe all the comparison hypersurfaces that we need  in the following. 
In Section \ref{proof}, we prove our main Theorem.  To simplify the reading, in the Appendix, we summarize the  principal notations introduced in the article. Throughout the article (except in  a part of Section \ref{hor-cyl}) we will use the ball model for the hyperbolic space. Moreover we will use extensively Alexandrov reflection method which is explained in details in \cite[Chapter 7]{H},\cite{N}. 

\section{The Lima\c con in the hyperbolic space}\label{sez-lima}

In this section we describe a family of hypersurfaces of the hyperbolic space $\hh^n,$ analogous to the classical  curve in the Euclidean plane known as {\em Lima\c con  de Pascal}. Such hypersurfaces will be convenient  to estimate radii of disks appearing in the proof of the main theorem.

\begin{Definition} \label{def lima} We call \emph{hyperbolic Lima\c con} the hypersurface $\hl$ of $\hh^n$ defined in the following way. Fix two points $A\neq C\in\hh^n$ and a positive constant $c$, let $\cc$ be the geodesic sphere of center $C$ and radius $c$. For any $P\in\cc$ let $A_P$ be the reflection of $A$ across the totally geodesic hyperplane of $\hh^n$ tangent to $\cc$ at $P$. Then $\hl=\{A_P\in\hh^2\ |\ P\in\cc\}$. See Figure 1 below.
\end{Definition}

\begin{figure}[!h]
\centering
\includegraphics[width=0.7\textwidth]{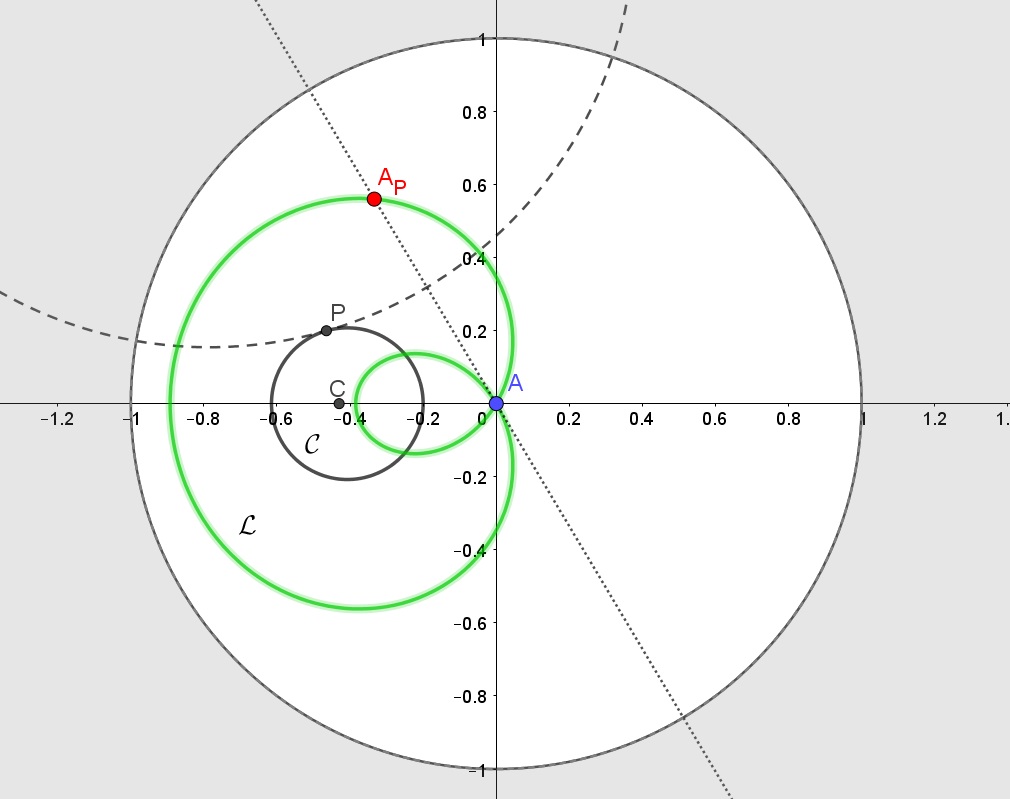}
\caption{Hyperbolic Lima\c con for $n=2$, $a>c$.}
\end{figure}

In what follows we want to describe the main proprieties of $\hl$. It is evident by definition that $\hl$ is invariant under rotations around the geodesic passing through $A$ and $C$, therefore it is enough to study only the case $n=2$. 

\begin{Remark}
Since $\hh^2$ is  homogeneous, up to isometries of the ambient space, $\hl$ depends only on two positive parameters: $a:=d(A,C)$ and $c$.
\end{Remark}

In the following $a$ and $c$ will be called the {\em parameters} of the Lima\c con and $A$ the {\em base point} of it.

\begin{Remark}
The name of this curve is justified by the fact that the same construction in the Euclidean plane produces the classical Lima\c con parametrized, up to isometry, by
$$
\hl(\vartheta)=\left(-a\cos(2\vartheta)+2c\cos(\vartheta),-a\sin(2\vartheta)+2c\sin(\vartheta)\right).
$$
Moreover it is well known that the Euclidean Lima\c con can be defined in many equivalent ways. For example it is the pedal curve of a circumference. The pedal curve of a circumference in the hyperbolic plane is the subject of \cite{SJ}.
\end{Remark}

The following result says that  the shape of  the hyperbolic Lima\c con is  qualitatively analogous to  the shape of the Euclidean one.

\begin{Lemma}\label{lemma lima}

Let $\hl$ be the hyperbolic Lima\c con with parameters $a,$ $c$ and  base point $A.$ Let $\mathcal C$ be the sphere defining $\hl$ and $C\not=A$ its center.
Then $\hl$ is a closed continuous curve, which continuously depends on $a$ and  $c.$  It is symmetric with respect to the geodesic through $A$ and $C$. If $a<c$ it is a simple curve. If $a=c$ it has a cusp in $A$. If $a>c$ it has two loops, one inside the other, and it crosses itself only in $A.$

\end{Lemma}

\proof
From the very definition it is evident that $\hl$ is an immersion of a $\mathbb S^1$ in $\hh^2$, that  continuously depends on the parameters $a$ and $c$ and  is symmetric with respect to the complete geodesic joining $A$ and $C$. 

Note that, for any choice of $a$ and $c$, $\hl$ has no multiple points except, possibly, $A$. In fact take any point $B\in\hl$, $B\neq A$, let $\overline{AB}$ the geodesic segment joining $A$ and $B$ and let $M$ be its middle point. Reversing the construction of Definition \ref{def lima}, the geodesic passing through $M$ and orthogonal to $\overline{AB}$ has to be tangent to $\cc$. Let $X$ be the tangency point, then $B=A_X$, and $X$ is uniquely determined by this procedure.  Moreover, by Definition \ref{def lima}, we have that $A\in\hl$ if and only if there is a geodesic tangent to $\cc$ containing $A$.
After these general facts we need to distinguish the three cases.

If $a<c$ the point $A$ lies in the disk bounded by $\cc$.  In this case $A\notin\hl$, otherwise there would be a geodesic passing through $A$ and tangent to $\cc$. Therefore $\hl$ is a simple curve.

When $a>c$, $A$ is clearly outside the disk bounded by $\cc$, then there are two distinct geodesic passing though $A$ and tangent to $\cc$. Hence $A$ is a double point for $\hl$ and, as showed above, it is its only multiple point. Now let $\gamma$ be the complete geodesic joining $A$ and $C$: $\gamma$ intersects $\cc$ in exactly two points: denote by $P$ the closest to $A$, by $Q$ the second one and let $A_{P}$ and $A_{Q}$ be the corresponding points on $\hl$. Since $\gamma$ and $\cc$ are orthogonal in $P$ and $Q$, we have that $\gamma\cap\hl=\{A,A_{P},A_{Q}\}$. Moreover, walking along $\gamma$ from $A$ in the direction of $C$, we will meet $A$, $A_P$ and $A_Q$ exactly in that order. Suppose now to walk along $\cc$ clockwise starting from $P$, then in $\hl$ we will meet $A_P$, $A$, $A_Q$ and $A$ again exactly in that order. Since $\hl$ is continuous and by the order of the points $A$, $A_P$ and $A_Q$ on $\gamma$, it follows that $\hl$ has two loops branching in $A$, one inside the other. By construction $d(A,A_P)=2d(A,P)$ and $d(A,A_Q)=2d(A,Q)$, therefore the smaller loop is that one passing through $A_P$. 

The case $a=c$ can be thought as the limit case when $a>c$ and $a$ converges to $c$: in this case $P$ converges to $A$, then the smaller loops shrinks to $A$ producing a cusp. 

\cvd

For applications we would like to estimate the size of $\hl$ with particular attention to the case of a small $c$.

\begin{Lemma}\label{stima loop}
Let $\hl$ be the hyperbolic lima\c con with parameter $a>c$ and base point $A$. Let $\cc$ be the sphere defining $\hl$, let $C$ be its center and let $P$ be the point of $\cc$ closest to $A$. Then we have:

\begin{enumerate}
\item The smaller (resp. the larger) loop of $\hl$ is contained in (resp. contains) the disk with center $P$ and radius $a-c$. 
\item If $a>2c$, then the smaller loop of $\hl$ bounds the disk with center $C$ and radius $a-2c$.
\item The whole $\hl$ is contained in the disk with center $C$ and radius $a+2c$.
\end{enumerate}
Moreover, for any fixed $a$, $C$ and $A$, when $c$ converges to zero, $\hl$ converges to twice the sphere with center $C$ and radius $a$.
\end{Lemma}

\proof Since $a>c$, by  Lemma \ref{lemma lima}, $\hl$ has two loops. 

\begin{enumerate}
\item Let$A_P$ be as in the proof of Lemma \ref{lemma lima} and let ${\cc_1}$ be the geodesic sphere of center $P$ and radius $a-c$.  We claim that ${\cc_1}\cap\hl=\{A,\ A_P\}$. In fact $A,A_P\in{\cc_1}\cap\hl$ trivially. For any point $B\in{\cc_1}$ with $B\neq A,A_P$, let $\overline{AB}$ be the geodesic segment joining $A$ and $B$ and let $M$ be its middle point. See Figure 2 for a picture of the construction. For the choice of $B$ we have that $M\neq P$. Let ${\gamma}$ be the geodesic through $M$ and $P$. The geodesic triangles $AMP$ and $BMP$ are congruent, hence ${\gamma}$ is orthogonal to $\overline{AB}$ in $M$. Now suppose that there exists a $B\in\hl\cap{\cc_1}$ with $B\neq A,A_P$, then, by the construction of Definition \ref{def lima}, ${\gamma}$ should be tangent to $\cc$, but we know that $P\in\cc\cap{\gamma}$, hence we would get $B=A_P$, having a contradiction. From this claim it follows that the smaller (resp. the larger) loop of $\hl$ is inside (resp. outside) ${\cc_1}$.


\item Now suppose that $a>2c$. Let $\cc_2$ be the geodesic sphere with center $C$ and radius $a-2c$. We want to prove that the smaller loop of $\hl$ bounds $\cc_2$. By construction we have $A_P\in\cc_2\cap\hl$. We claim that $\cc_2\cap\hl=\{A_P\}$. 
In fact, fix $X\in\cc_2$, $X\neq A_P$, let $\gamma$ be the unique geodesic tangent to $\cc$ and orthogonal to the geodesic $\alpha$ containing  the segment $\overline{XA}$, let $Y=\gamma\cap\cc$ be the tangency point and let $Z=\gamma\cap\alpha$. See Figure 2 for a picture of the construction.  In order to prove that $X\notin\hl$, by Definition \ref{def lima}, we need to prove that $Z$ is not the middle point of $\overline{XA}$.  If  $Z\not\in\overline{XA},$ we are done, hence  we  assume  $Z\in\overline{XA}.$ Since $X\neq A_P$ we have that the following triangle inequality is strict:
$$
d(X,Y)<d(X,C)+d(C,Y)=a-c.
$$
On the other hand, since $Y\in\cc$ and by definition of $P$, we have 
$$
d(A,Y)\geq d(A,P)=a-c.
$$
Therefore 
\begin{equation}\label{eq02}
d(X,Y)<d(A,Y).
\end{equation}
Now let us consider the right-angled hyperbolic triangles $XYZ$ and $AYZ$. Let $\beta$ (resp. $\beta'$) be the angle of $XYZ$ (resp. $AYZ$) with vertex $X$ (resp. $A$). By \eqref{eq02} and some hyperbolic trigonometry - see for instance Theorem 7.11.2 of \cite{Be} - we get
\begin{eqnarray*}
\tanh^2d(X,Z) & =& \tanh^2d(X,Y)\cos^2\beta\\
 & =& \tanh^2d(X,Y)\left(1-\frac{\sinh^2d(A,Y)}{\sinh^2d(X,Y)}\sin^2\beta'\right)\\
&<& \tanh^2d(A,Y)\cos^2\beta'\ =\ \tanh^2d(A,Z).
\end{eqnarray*}
It follows that 
$$
d(X,Z)<d(Z,A),
$$
hence $Z$ cannot be the middle point of $\overline{XA}$ and therefore $X\notin\hl$. The result follows noticing that $\cc_2$ is a continuous curve and that the antipodal point of $A_P$ in $\cc_2$ is in the compact domain bounded by the smaller loop of $\hl$, therefore the closed disk bounded by $\cc_2$ is inside this domain as well.

\begin{figure}[!h]
\centering
\includegraphics[width=0.6\textwidth]{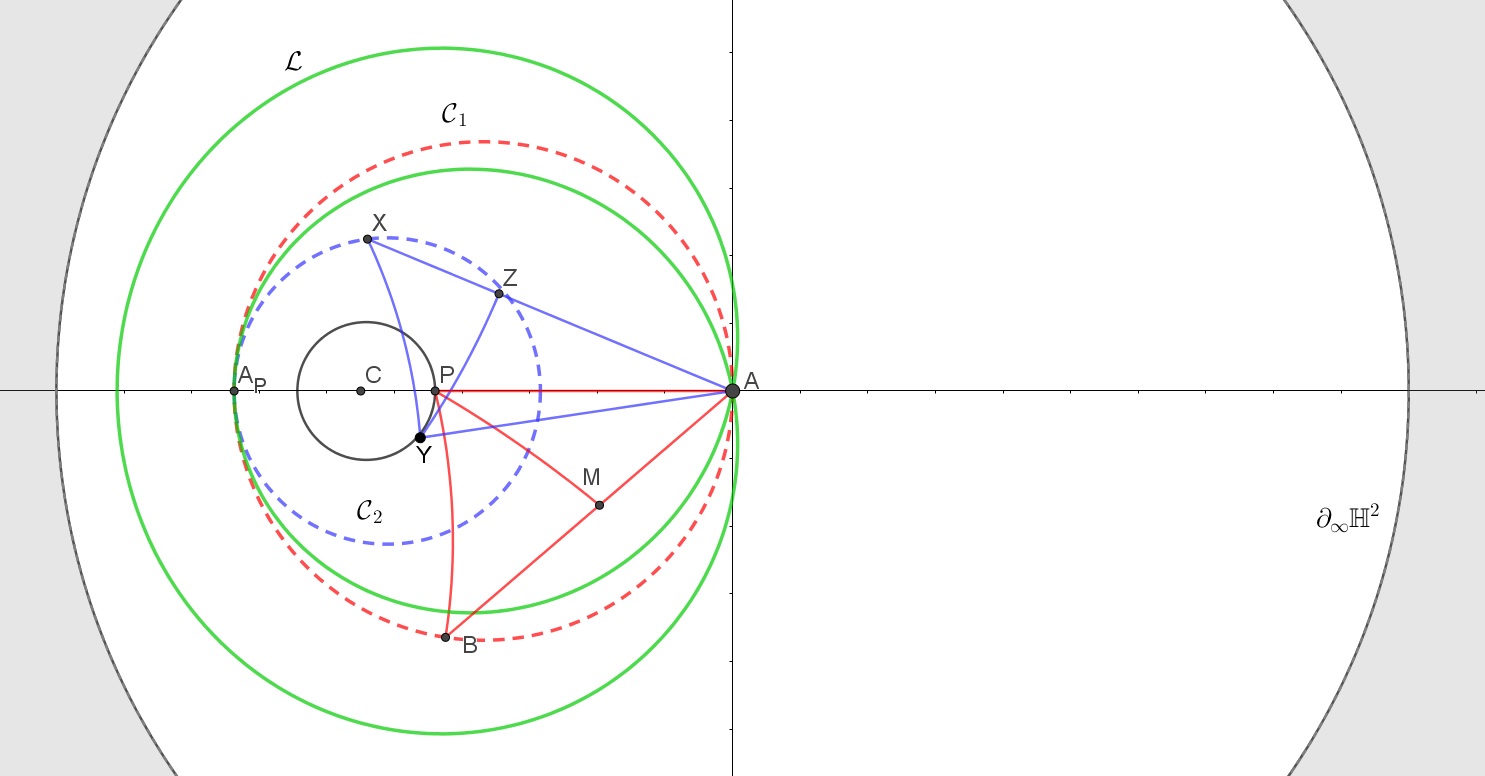}
\caption{Estimate of the diameter of the smaller loop. Green: hyperbolic Lima\c con with $a>2c$. Red: the construction in the proof of Lemma \ref{stima loop}-1. Blue: the construction in the proof of Lemma \ref{stima loop}-2.}
\end{figure}

\item The proof is similar to that of the previous case but, this time, the inequalities are reversed. Let $\mathcal C_3$ be the geodesic sphere with center $C$ and radius $a+2c$, let $X$ a point of $\mathcal C_3$ and define $Y$ and $Z$ analogously to the previous case. We have that $d(X,Y)>a+c>d(A,Y)$. Therefore either $Z\notin\overline{XA}$ or $d(X,Z)>d(Z,A)$. In any case $Z$ is not the middle point of $\overline{XA}$, hence $X\notin\hl$. 
\end{enumerate}
Finally, when $c$ converges to $0$, we have that $P$ converges to $C$, hence for any $i=1,2,3$ we get that $\mathcal C_i$ converges to the geodesic sphere with center $C$ and radius $a$. \cvd

\section{Constant mean curvature comparison hypersurfaces}\label{sez-barriere}

There are many examples of constant mean curvature hypersurfaces  in $\hh^n\times\rr$ invariant by some ambient isometry. 
In this section we describe some of them, that  will be mainly  used as barrier in the proof of our main Theorem. 

\subsection{Rotationally invariant hypersurfaces}\label{rotaz}
P. B\'erard and R. Sa Earp \cite{BS} classify the rotationally invariant hypersurfaces of $\hh^n\times\rr$ with constant mean curvature. Following their notations, for any $m\in\mathbb N$, for any $H>0$ and for suitable choice of the parameter $d,$ we define
\begin{eqnarray}
\nonumber I_m(t) &=& \int_0^t\sinh^m(\tau) d\tau;\\
\label{rot alta} \lambda_{H,d}(\rho) & =& \int_{\rho_0}^{\rho}\frac{nHI_{n-1}(t)+d}{\sqrt{\sinh^{2n-2}(t)-(nHI_{n-1}(t)+d)^2}}dt,
\end{eqnarray}
where $\rho_0\geq 0$ is the infimum of the interval where the integrand function exists. The rotation around the axis $\{0\}\times\rr$ of the graph of the function $\lambda_{H,d}$ produces, up to isometry of the ambient space, all the rotationally invariant hypersurfaces with constant mean curvature $H$. The case of $n=2$ has been studied in detail in \cite{On, ST, NSST}. 

In this section we describe only the three types of hypersurfaces in this class, that are most relevant for our purposes. We refer to the bibliography cited above for an exhaustive discussion about the topic.

\textbf{Case 1: the spheres $\mathcal S_H$.}\\
 When $H>\frac{n-1}{n}$ and $d=0$, $\lambda_{H,0}$ is defined for $\rho\in[0,R_{\mathcal S}]$ where $R_{\mathcal S}$ is the solution of the equation
\begin{equation}\label{rs}
\sinh^{n-1}(\rho)-nHI_{n-1}(\rho)=0.
\end{equation}

 The graph of the function $\lambda_{H,0}$ is tangent to the plane $t=0$ when $\rho=0$ and it has a vertical tangent at $\rho=R_{\mathcal S}$. Let $h_{\mathcal S}=\lambda_{H,0}(R_{\mathcal S})$ be the maximal height of such curve. Let $\mathcal S_H$ be the hypersurface generated by rotating the graph of $\lambda_{H,0}$ and the graph of the function $-\lambda_{H,0}+2h_{\mathcal S}$ around the $t$-axis. According to  \cite[Theorem 2.3]{BS}, $\mathcal S_H$ is a compact embedded smooth hypersurface with the topology of the sphere. When $n=2$ - hence $H>\frac 12$ - we have the explicit expression
\begin{align*}
&\lambda_{H,0}(\rho)=\frac{4H}{\sqrt{4H^2-1}}\arcsin\frac{1}{2H}-\frac{4H}{\sqrt{4H^2-1}}\arctan\sqrt{\frac{1-4H^2\tanh^2\frac{\rho}{2}}{4H^2-1}}, \\
 &\rho\in \left[0, \cosh^{-1}\left(\frac{4H^2+1}{4H^2-1}\right)\right] \notag
 \end{align*}

\textbf{Case 2: the complete graphs $\mathcal S_{\frac{n-1}{n}}$.} \\
When $H=\frac{n-1}{n}$ and $d=0$, the curve $\lambda_{H,0}$ is defined for any $\rho\geq 0$. Let $\mathcal S_{\frac{n-1}{n}}$ be the hypersurface generated by rotating $\lambda_{\frac{n-1}{n},0}$ around the $t$-axis. According to  \cite[Theorem 2.1]{BS}, $\mathcal S_{\frac{n-1}{n}}$ is a simply connected entire vertical graph, contained in a half-space and tangent to the hyperplane $t=0$ at $\rho=0$. Moreover when $n=2$  \eqref{rot alta} can be solved explicitly and one has
\begin{equation*}
\lambda_{\frac{1}{2},0}(\rho)=2\left(\cosh\frac{\rho}{2}-1\right). 
\end{equation*}

\textbf{Case 3: immersed annuli $\mathcal A_{d}$.}\\
 When $H=\frac{n-1}{n}$ and $d<0$, the curve $\lambda_{\frac{n-1}{n},d}$ is defined for any $\rho\geq r_{0,d}$ where $r_{0,d}$ is the unique solution of the equation 
\begin{equation}\label{r0d}
\sinh^{n-1}(\rho)+(n-1)I_{n-1}(\rho)+d=0.
\end{equation}
Moreover $\lambda_{\frac{n-1}{n},d}$ has a vertical tangent at $r_{0,d}$, it is negative for $\rho$ close to $r_{0,d}$ and $\lim_{\rho\rightarrow+\infty}\lambda_{\frac{n-1}{n},d}(\rho)=+\infty$. From \eqref{rot alta} it easy to see that $\lambda_{\frac{n-1}{n},d}$ has only one critical point at $\rho=r_{1,d}$ where $r_{1,d}$ is the unique solution of the equation
\begin{equation}\label{r1d}
(n-1)I_{n-1}(\rho)+d=0.
\end{equation}
Let $\mathcal A_{d}$ be the hypersurface generated by rotating the graphs of $\lambda_{\frac{n-1}{n},d}$ and $-\lambda_{\frac{n-1}{n},d}$ around the $t$-axis. According to \cite[Theorem 2.1]{BS}, $\mathcal A_{d}$ is a complete hypersuface, symmetric with respect to the hyperplane $t=0$ with self intersection along a sphere of this hyperplane. Moreover $\mathcal A_{d}\cap\rr^{\pm}$ are vertical graphs outside a disk of $\hh^n\times\{0\}$ with center in the origin and radius $r_{0,d}$.

\begin{Remark}
\label{no-flat}{\em
We sum up some important relations between the hypersurfaces described above.
\begin{itemize}
\item When $d\longrightarrow 0$, then hypersurfaces  $\mathcal A_{d}$ tend  to  the union of $\mathcal S_{\frac{n-1}{n}}$ with its reflection with respect to the slice $\{t=0\}.$ 
\item When $H\longrightarrow\frac{n-1}{n},$ the spheres ${\mathcal S}_H$ tend to the complete graph ${\mathcal S}_{\frac{n-1}{n}}.$ 
Notice that, differently  from the Euclidean case, here, the spheres do not converge to a horizontal slice.
\end{itemize}}
\end{Remark}

For future use we need to estimate various quantities associated to the hypersurfaces $\mathcal S_{H}$, $\mathcal S_{\frac{n-1}{n}}$ and $\mathcal A_{d}$. In particular we would like to compare $r_{0,d}$ and $r_{1,d}$ defined above with $\rho_H^*$ defined in the following way: fix $d$ and $H\geq\frac{n-1}{n}$, $\rho_H^*$ is the radius of $\mathcal S_{H}$ at height $h^*$, where $h^*$ is defined in \eqref{h*} below and it is a suitable approximation of the height of the portion of $\mathcal A_d$ between $r_{0,d}$ and $r_{1,d}$. In particular $\rho_H^*$ satisfies the equation 
\begin{equation}\label{rH*}
\lambda_{H,0}(\rho_H^*)=h^*.
\end{equation}

\begin{Lemma}\label{raggi D}
With the notations introduced so far, the following hold for any $H\geq\frac{n-1}{n}$:
\begin{equation*}
\lim_{d\rightarrow 0}r_{0,d}=0,\quad\lim_{d\rightarrow 0}r_{1,d}=0,\quad\lim_{d\rightarrow 0}\rho_H^*=0;
\end{equation*}
moreover 
\begin{equation*}
\lim_{d\rightarrow 0}\frac{r_{0,d}}{r_{1,d}}=0,\quad\lim_{d\rightarrow 0}\frac{r_{1,d}}{\rho_H^*}=0.
\end{equation*}
\end{Lemma}
\proof The first two limits follow directly by the definition of $r_{0,d}$ and $r_{1,d}$. Now note that, by standard computations, we have that for any $m\in\mathbb N$
\begin{equation}\label{lim base}
\lim_{t\rightarrow 0}\frac{(m+1)I_m(t)}{t^{m+1}}=1.
\end{equation}
It follows that 
\begin{equation*}
\lim_{d\rightarrow 0}\frac{r_{1,d}}{|d|^{\frac 1n}}= \left(\frac{n}{n-1}\right)^{\frac 1n},\quad
\lim_{d\rightarrow 0}\frac{r_{0,d}}{|d|^{\frac{1}{n-1}}}= 1.
\end{equation*}
Therefore we get $\lim_{d\rightarrow 0}\frac{r_{0,d}}{r_{1,d}}=0.$

Using again \eqref{lim base}, $\lambda_{\frac{n-1}{n},d}(r_{1,d})$ can be estimate for $|d|$ very small as follows:
\begin{eqnarray}
\label{h*} \lambda_{\frac{n-1}{n},d}(r_{1,d})  &\leq&  -h^*:=\int_{r_{0,d}}^{r_{1,d}}\frac{(n-1)I_{n-1}(t)+d}{\sinh^{n-1}(t)}dt \\
\nonumber&\approx& \int_{r_{0,d}}^{r_{1,d}}\frac{\frac{n-1}{n}t^n+d}{t^{n-1}}dt\\
\nonumber&=&\frac{n-1}{n}(r_{1,d}^2-r_{0,d}^2) +\frac{d}{2-n}(r_{1,d}^{2-n}-r_{0,d}^{2-n})\\
&\approx& -\frac{1}{n-2}\left(\frac{n-1}{n}\right)^{\frac{1}{n-1}} r_{1,d}^{\frac{n}{n-1}},\label{eq:rh}
\end{eqnarray}

By the use of \eqref{lim base} and definition of $h^*$ \eqref{h*}, when $|d|$ is close to zero we have
\begin{eqnarray*}
h^*&=&\int_0^{\rho_H^*}\frac{nHI_{n-1}(t)}{\sqrt{\sinh^{2n-2}(t)-n^2H^2I_{n-1}^2(t)}}dt\\
&\approx&H\int_0^{\rho_H^*}tdt=\frac H2(\rho_H^*)^2,
\end{eqnarray*}
Together with \eqref{eq:rh} this implies that, as $d\rightarrow 0$, $\rho_H^*\approx c_{n,H}r_{1,d}^{\frac{n}{2(n-1)}}$ where $c_{n,H}=\sqrt{\frac{2}{H(n-2)}}\left(\frac{n-1}{n}\right)^{\frac{1}{2(n-1)}}$. Hence the results involving $\rho_H^*$ follow easily.
\cvd

\subsection{Horizontal cylinders}
\label{hor-cyl}
In \cite{BS}, P. B\'erard and R. Sa Earp  describe constant mean curvature hypersurfaces of $\hh^n\times\rr$ which are invariant under hyperbolic translations of $\hh^n.$ 
Following their notations, for any $m\in\mathbb N$, for any $H>0$ and for suitable choice of the parameter $d,$ we define
\begin{eqnarray}
\nonumber J_m(t) &=& \int_0^t\cosh^m(\tau) d\tau;\\
 \mu_{H,d}(\rho) & =& \int_{\rho_0}^{\rho}\frac{nHJ_{n-1}(t)+d}{\sqrt{\cosh^{2n-2}(t)-(nHJ_{n-1}(t)+d)^2}}dt,\label{mu}
\end{eqnarray}
where $\rho_0\geq 0$ is the infimum of the interval where the integrand function exists.





The graphs of the functions $\mu_{H,d}$ are the generating curves of hypersurfaces constructed as follows. Let $\gamma$  be a complete geodesic through the origin  of the hyperbolic space $\hh^n\times\{0\},$ parametrized by the signed distance $\rho$ to the origin. Let $\pi$ be the hyperplane of $\hh^n\times\{0\}$ orthogonal to $\gamma$ at the origin. Consider the curve $(\rho,\mu_{H,d}(\rho))$ embedded in the plane $\gamma\times\mathbb R$. For any $\rho$ let $\pi_{\rho}$ be the vertical translation of $\pi$ in the slice $\hh^{n}\times\{\mu_{H,d}(\rho)\}$. The desired hypersurface is obtained translating each point $(\rho,\mu_{H,d}(\rho))$ along any geodesic of $\pi_{\rho}$ passing through the origin of $\pi_{\rho}$. By \cite{BS} such hypersurface has constant curvature $H$.

The case of $n=2$ has been studied in detail in \cite{On, Man, ST, NSST}. 

In this paper we are only interest in the case $H>\frac{n-1}{n}$ and $d=0$. For any fixed $H$ we denote with $\mathcal C_H$ such hypersurface. The value of $n$ will be clear from the context. The function $\mu_{H,0}$ is defined for $\rho\in[0,R_{\mathcal C}]$, where $R_{\mathcal C}$ is the unique solution of the equation
\begin{equation}\label{rc}
\cosh^{n-1}(\rho)-nHJ_{n-1}(\rho)=0.
\end{equation}
When $n=2$, $\mathcal C_H$ was explicitly parametrized by J. M. Manzano in \cite{Man}: using the half-space model for $\hh^2$ we have that, up to isometry of the ambient space, 

\begin{equation*}
\mathcal C_H(u,v)=\left(\frac{e^v\sin u}{\sqrt{4H^2-1}}, e^v,\frac{2H}{\sqrt{4H^2-1}}\tan^{-1}\frac{\cos u}{\sqrt{4H^2-1+\sin^2 u}}\right)
\end{equation*}
where $u,v\in\rr$. In this case it is evident that $\mathcal C_H$ has the topology of the cylinder and it is a bi-graph on the non compact domain of $\hh^2$  bounded by the two equidistant curves of constant geodesic curvature $-\frac{1}{2H}$
$$
\mathcal C_H\cap\{t=0\}=\left\{\left(\left.\frac{\pm e^v}{\sqrt{4H^2-1}},e^v,0\right)\right|v\in\rr\right\}.
$$ 
The distance between these two curves is $2R_{\mathcal C}$, hence fix $v=0$, then we can compute
\begin{equation}
R_{\mathcal C}=\frac 12d_{\hh^2}\left(\left(\frac{-1}{\sqrt{4H^2-1}},1\right), \left(\frac{1}{\sqrt{4H^2-1}},1\right)\right)=\frac 12\ln\left(\frac{2H+1}{2H-1}\right).\label{rc2}
\end{equation}


\begin{figure}[!h]
\centering
\includegraphics[width=0.8\textwidth]{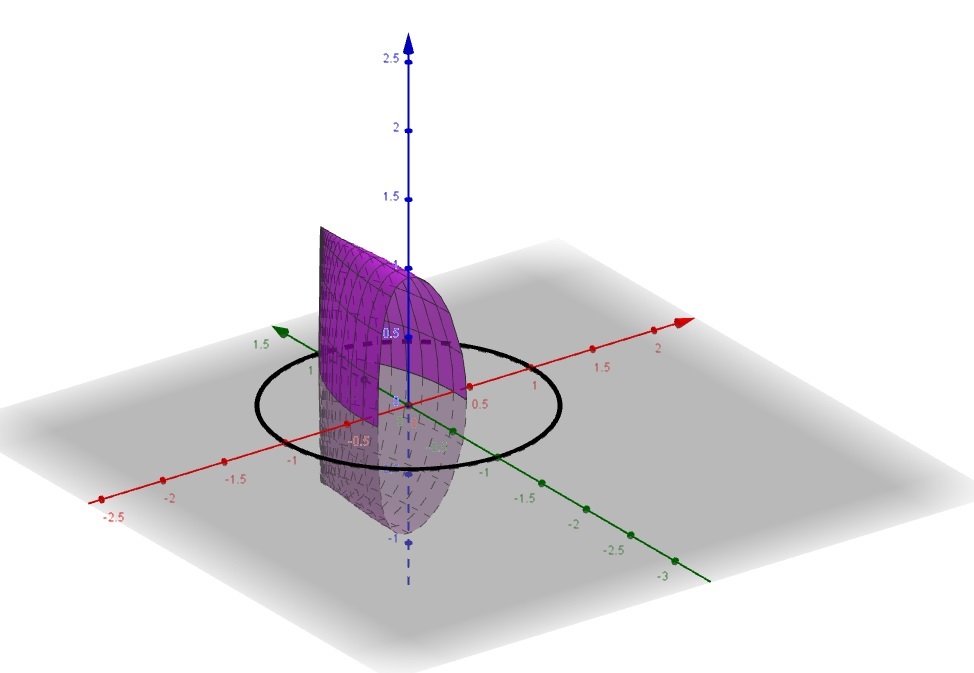}
\caption{Horizontal cylinder $\mathcal C_H$ in $\hh^2\times\rr$ in the disk model, H=0,77.}
\end{figure}

When $n\geq 3$ we do not have an explicit parametrization for $\mu_{H,0}$, but Theorem 2.4, part 1 of \cite{BS} can be easily extended to the case $H>\frac{n-1}{n}$ proving that $h_{\mathcal C}:=\lim_{\rho\rightarrow R_{\mathcal C}}\mu_{H,0}(\rho)$ is finite, that $\mu_{H,0}$ is  a strictly increasing and convex function with a vertical tangent at $\rho=R_{\mathcal C}$. Therefore the reflections of $\mu_{H,0}$  with respect to $\rho=0$ and $t=h_{\mathcal C}$ produce a compact strictly convex simple curve. It follows that the hypersurface $\mathcal C_{H}$ is embedded and it is symmetric with respect to a horizontal hyperplane and the parts above and below this hyperplane are vertical graphs. Moreover  
$\mathcal C_{H}$ has the topology of $\mathbb S^1\times\mathbb R^{n-1}.$ 
Note that $\lim_{H\rightarrow\frac{n-1}{n}}R_{\mathcal C}$ is infinite for $n=2$ and finite otherwise. 

We would like to conclude this section comparing the radius of the cylinder $\mathcal C_H$ and of the compact sphere $\mathcal S_H$ with the same mean curvature.

\begin{Lemma}\label{rs vs rc}
For any $n\geq 2$ and for any $H>\frac{n-1}{n}$ we have
$$
R_{\mathcal S}>R_{\mathcal C},
$$
where $R_{\mathcal S}$ is the solution of equation \eqref{rs} and $R_{\mathcal C}$ is the solution of equation \eqref{rc}.
\end{Lemma}
\proof 
We start by  introducing the following iteration formula taken from \cite{BS}:
\begin{equation}\label{iter J}
\left\{\begin{array}{rcl}
J_0(x)&=&x;\\
J_1(x)&=&\sinh(x);\\
mJ_m(x)&=&\sinh(x)\cosh(x)^{m-1}+(m-1)J_{m-2}(x),\quad \forall m\geq 2;
\end{array}\right.
\end{equation}

Notice that $J_m(x)>0$ for any $m$ and any $x>0$. Fix $n$ and $H$ as in the statement. Let us define the real functions
\begin{eqnarray*}
\varphi(x) & =& \sinh^{n-1}(x)-nHI_{n-1}(x);\\
\psi(x) & =& \cosh^{n-1}(x)-nHJ_{n-1}(x).
\end{eqnarray*}

By trivial arguments we have that the  unique strictly positive critical point of $\varphi$ is $\tilde R:=\tanh^{-1}\left(\frac{n-1}{nH}\right)$ and that $\tilde R< R_{\mathcal S}$.

When $n=2$ by \eqref{rc2} we have that $\tilde R=R_{\mathcal C}$ and the result holds.

Now let $n\geq 3$. By \eqref{iter J} and definition of $\tilde R$ we have that
\begin{eqnarray*}
\psi(\tilde R) & =& \cosh^{n-1}(\tilde R)-\frac{nH}{n-1}\left(\sinh(\tilde R)\cosh^{n-2}(\tilde R)+(n-2)J_{n-3}(\tilde R)\right)\\
&<&\cosh^{n-2}(\tilde R)\left(\cosh(\tilde R)-\frac{nH}{n-1}\sinh(\tilde R)\right)\\
&=&0=\psi(R_{\mathcal C})
\end{eqnarray*}

where in the inequality, we use that  $J_{n-3}>0.$ 
 Therefore $R_{\mathcal C}<\tilde R$ holds because $\psi$ is strictly decreasing. 
\cvd

\section{Main Theorem}\label{proof}

In this Section, we prove the generalization of Ros-Rosenberg theorem, stated in the Introduction.

A closed hypersurface $\Gamma$ in the hyperbolic space $\hh^n$  is said \emph{horoconvex} if all the principal curvatures of $\Gamma$ are strictly larger than $1$. Given a horoconvex hypersurface $\Gamma$, we call the exterior (resp. interior) radius of $\Gamma$ the minimum (resp. maximum) of the radii $\rho$ such that for any $p\in\Gamma$ there exists a geodesic sphere $S$ with radius $\rho$ tangent to $\Gamma$ in $p$ and such that $\Gamma$ lies in (resp. encloses) the closed disk bounded by $S$. We denote by $r_{ext}$ the exterior radius and by $r_{int}$ the interior radius. It is clear that, for any $\Gamma$, we have $r_{ext}\geq r_{int}$ and that the equality holds if and only if $\Gamma$ is a geodesic sphere of radius $r_{ext}$. 

The main result of the paper is the following.

\begin{Theorem}\label{theorem-main}
Let $M$  be a $n$-dimensional hypersurface of $\hh^n\times\rr$ embedded in  $\hh^n\times]0,\infty]$ with constant mean curvature $H>\frac{n-1}{n}$ and boundary $\partial M=\Gamma$. Assume that $\Gamma$ is a closed $(n-1)$-dimensional horoconvex hypersurface of the slice  $P=\hh^n\times\{0\}$ satisfying the pinching $2r_{int}>r_{ext}$. Then there is a constant $\delta(n,H)>0$ depending only on $n$ and $H$ such that if $r_{ext}\leq\delta(n,H)$, then $M$ is topologically a disk.
Moreover either  $ M$ is a graph over the domain $\Omega$ bounded by $\Gamma$ or $N= M\cap(\Omega\times]0,\infty[)$ is a graph over $\Omega$ and $ M\setminus N$ is a graph over $\Gamma\times]0,\infty[$ with respect to the lines orthogonal to $\Gamma\times]0,\infty[.$
\end{Theorem}

\begin{Remark}{\em  In the case  $\Gamma$ is a geodesic sphere and $M$ satisfies the assumption of Theorem \ref {theorem-main}, it is easy to prove that $M$ is rotationally symmetric, by the use of Alexandrov reflections with respect to vertical hyperplanes. Hence $M$ is a portion of  a vertical translation of the hypersurface $\mathcal S_H,$ defined in Section \ref{rotaz}.}
\end{Remark}

From now on, we suppose that $\Gamma$ is not a geodesic sphere, hence $r_{ext}>r_{int}.$  The strategy of the proof of Theorem \ref{theorem-main} is inspired by the one in \cite{NS, S} and it is divided into two cases according to the following definition.

\begin{Definition}\label{def short tall}{\em
A hypersurface $M$ of $\hh^n\times\rr$ with constant mean curvature $H>\frac{n-1}{n}$ is called {\em short} if it is contained between   two slices at distance smaller  than the height of the cylinder $\mathcal C_H$ with the same mean curvature,  defined in Section \ref{hor-cyl}. $M$ is called {\em tall} otherwise.}
\end{Definition}


When $M$ is short, the proof of Theorem \ref{theorem-main} is more direct because shortness yields that $M$ is a graph over $\Omega$. When $M$ is tall, we will prove that $M$ is a union of pieces, each one graph in some system of coordinates.

In \cite{RR}  the authors introduced the notion of small surface with constant mean curvature $H$: a hypersurface contained in a ball of mean curvature larger than $H$. Here, the vertical and the horizontal directions are not homogeneous, hence the suitable notion is that of short hypersurface.

The first  step, common to both cases short and tall, is to show that  in a small vertical cylinder, $M$ is a graph.  
This will be proved in the following Lemma, that is  the analogous of Lemma 3 of \cite{RR}. 
Differently from \cite{RR}, we need a quantitative estimate of the radius of the cylinder, that  will be evaluated using the Lima\c con described in Section \ref{sez-lima} .

In the following, we let $W$ be the domain in $\hh^n\times{\mathbb R}$ bounded by $M$ and $\Omega$ and we will denote by $D(\rho)$ any disk with radius $\rho$ in a  slice $\hh^n\times\{t\}$,  where the value of $t$ and the center of $D(\rho)$ will be clear from the context.

\begin{Lemma}\label{lemmaRR}
Let $M$ and  $\Gamma$  be  as in the statement of Theorem \ref{theorem-main}. Then, there exists a disk $D(r_{min})$  in  $\hh^n\times\{0\},$ where $r_{min}$     depends only of the principal curvatures of $\Gamma=\partial M,$ such that $M\cap (D(r_{min})\times{\mathbb R})$ is a graph. In particular, since $\Gamma$ satisfies the pinching $2r_{int}>r_{ext}$, then $2r_{int}-r_{ext}< r_{min}< r_{int}$ holds. 
\end{Lemma}

\proof
Consider Alexandrov reflection with horizontal hyperplanes coming down. If we can arrive to $P:=\hh^n\times\{0\}$ without having a contact point between $M$ and its reflection, then $M$ is a graph over $\Omega$ and the result holds. Otherwise there is a height $t_0>0$ where the reflected hypersurface touches $\Gamma$ for the first time. Let $q\in\Gamma$ be the first touching point.  So $\{q\}\times]0,\infty)$ intersects $M$ exactly once and $\{q\}\times(0,2t_0)\subset int(W)$. Notice that the part of $M$ above $\hh^n\times\{t_0\}$ is a vertical graph.

Now let $v$ be a unit horizontal vector and $Q_v$ be a vertical hyperplane orthogonal to $v$. 
We will do Alexandrov reflection with the family of hyperplanes parallel to $Q_v.$
 Starting from  a $Q_v$ far away from $M$, we can move $Q_v$ parallel to itself  until it touches $M$ for the first time. Keep moving $Q_v$ and  let $M^*$ be the reflection, across $Q_v$, of  the part of $M$ behind $Q_v$. Continue moving $Q_v$ until there is a first touching point between $M$ and $M^*$. Since for any $p\in\Gamma$ the domain $\Omega$ bounds a disk of radius $r_{int}$ tangent at $p$ to $\Gamma$, then we could do Alexandrov reflections at  least until $Q_v$ is at distance $r_{int}$ from $\Gamma$. In order to avoid the dependence on the point $q$, we  stop to do reflection earlier, precisely when $Q_v$  becomes tangent to $\cc$, where $\cc$ is defined in the following way: let $\cc_{ext}$ be the geodesic sphere of $P$ with radius $r_{ext}$ and tangent to $\Gamma$ in $q$ which encloses $\Gamma$, then $\cc$ is the geodesic sphere with the same center of $\cc_{ext}$ and radius $r_{ext}-r_{int}$. 

For any unit horizontal vector $v$ let $q_v$ be the reflection of the point $q\in P$ with respect to $Q_v$ tangent to $\cc$. Let $\hl$ be set of all such points $q_v$. Since $Q_v$ is vertical and $Q_v\cap P$ is a totally geodesic hyperplane of the hyperbolic space $P$, then $q_v\in P$ and $\hl$ is a hyperbolic Lima\c con as in Definition \ref{def lima} with base point $q$ and parameters $a=r_{ext}$ and $c=r_{ext}-r_{int}$. Since $a>c$, $\hl$ has two loops by Lemma \ref{lemma lima}. Moreover, as shown in the proof of Lemma \ref{stima loop}, the smaller loop of $\hl$ is bounded by the sphere of radius $a-c=r_{int}$, tangent to $\Gamma$ in $q$, and such sphere is contained  in $\Omega$.

Furthermore, for any $p$ in the smaller loop of $\hl$, the vertical rectangle $\overline{qp}\times(0,2t_0)\subset int(W)$. The result will follow taking $r_{min}$ as the largest radius of a disk bounded by the smaller loop of $\hl$. Note that $r_{min}$ depends only on $a$ and $c$, i.e. only on the curvature of $\Gamma$, and not on $q$. 

Finally, since $\Gamma$ satisfies the pinching $2r_{int}>r_{ext}$, by Lemma \ref{stima loop} it holds that
$
2r_{int}-r_{ext}< r_{min}< r_{int}.
$
\cvd

\begin{Remark}\label{sigma}
\emph{By the construction described in the proof of Lemma \ref{lemmaRR} and by Lemma \ref{stima loop}-2, the centers of $\mathcal C_{ext}$, $\mathcal C$ and $D(r_{min})$ coincide. Up to isometry we can suppose that it is $(\sigma,0)\in\hh^n\times\rr$ where $\sigma$ is the origin of $\hh^n$.}
\end{Remark}

\begin{proof}[Proof of Theorem \ref{theorem-main}]

The strategy is to prove that $M$ is the union of components, each graphical above some domain.  At the end of the proof it will be clear that this determines the  fact that $M$ is topologically a disk.
Denote by $h_M$ the maximal height of $M$ above the plane $P:=\hh^n\times\{0\}.$ Fix any  point $A\in \Gamma$ and consider a disk $D(r_{ext})$ of radius $r_{ext}$ tangent to $\Gamma$ at $A.$  We divide the proof into two cases depending on whether $M$ is tall or short, as defined in  Definition \ref{def short tall}.

\vspace{12pt}

\begin{center} \noindent{\bf Case  1 ($M$ short):} $h_M<2h_{\mathcal C}.$
\end{center}

{\bf Claim 1A.} $M$ is contained in $D(r_{ext}+R_{\mathcal C})\times [0,2h_{\mathcal C}[$.\\

{  Consider the  horizontal cylinder ${\mathcal C}_H$ and let $\pi$ be hyperplane defined in Section 
\ref{hor-cyl}. 

Denote by $\Pi$ the vertical hyperplane containing $\pi.$ Denote by ${\mathcal C}^+_H$  the intersection of  ${\mathcal C}_H$ with one of the two halfspaces determined by $\Pi.$ 
Notice that, $\partial {\mathcal C}^+_H$ is the union of two hyperplanes,  each contained in a slice. Up to a vertical translation we may assume that  the lower boundary is on the slice $t=0$ and   the upper boundary  is  in the slice $t=2h_{\mathcal C}$.   
  By abuse of notation, we will call ${\mathcal C}^+_H$  any horizontal translation and any horizontal rotation of ${\mathcal C}^+_H.$}
Since $M$ is compact, we can  translate horizontally  ${\mathcal C}^+_H$ such that  $M\cap {\mathcal C}^+_H=\emptyset$ and $M$ lies in the side of ${\mathcal C}^+_H\cap\left(\hh^n\times [0,2h_{\mathcal C}]\right)$ which contains the axis of ${\mathcal C}_H$. Then, we translate ${\mathcal C}^+_H$  towards $M$ and  by the maximum principle,  ${\mathcal C}^+_H$ and $M$ cannot meet at an interior point. 
As $h_M<2h_{\mathcal C},$ one can translate ${\mathcal C}^+_H$ till its lower boundary on $t=0$ touches  the boundary of $D(r_{ext}).$ 
The same can be done for ${\mathcal C}^+_H$ with any horizontal axis in the slice $t=h_{\mathcal C},$ hence $M$ is contained in $D(r_{ext}+R_{\mathcal C})\times [0,2h_{\mathcal C}[$.\\

{\bf Claim 1B.} For $\Gamma$ sufficiently small,  $M$ is contained in the vertical cylinder above $\Omega$.\\
 By Lemma \ref{rs vs rc}, we know that, for any $n\geq 2$ and any $H>\frac{n-1}{n}$, $R_{\mathcal S}>R_{\mathcal C}$ holds. Choose $\Gamma$ small enough such that $R_{\mathcal S}>r_{ext}+R_{\mathcal C}.$ Let $\mathcal S_H$ be  the sphere with constant mean curvature $H$, translate it vertically such that $\mathcal S_H\subset\hh^n\times[-h_{\mathcal S},h_{\mathcal S}]$ and denote by ${\mathcal S}_H^+$ the part contained in $t>0$. Translate $\mathcal S_H^+$ vertically such that it is above $M$. Then translate ${\mathcal S}_H^+$ down. 
By the maximum principle, there can not be an  interior contact point between $M$ and  the translation of ${\mathcal S}_H^+.$ Moreover, as $M\subset D(r_{ext}+R_{\mathcal C})\times [0,2h_{\mathcal C}[$ and $R_{\mathcal S}>r_{ext}+R_{\mathcal C},$ the boundary of  ${\mathcal S}_H^+$ will not meet $M$ before coming back to $t=0$. Hence $M$ is below ${\mathcal S}_H^+$.
As $R_{\mathcal S}>r_{ext},$ by translating  horizontally ${\mathcal S}_H^+$ one can touch all the points of $\Gamma$ with ${\mathcal S}_H^+\cap\{t=0\}.$ By the maximum principle, $M$ stays below all such  translations of ${\mathcal S}_H^+.$ That is $M$ is contained in the vertical cylinder above $\Omega.$\\

This concludes the proof of Theorem \ref{theorem-main} in the case  $M$ short,  in fact since Claim 1B holds, using Alexandrov reflections with horizontal hyperplanes, it is easy to prove that   $M$ is a graph over $\Omega$, hence $M$ is topologically a disk. Moreover we have that $\delta(n,H)=R_{\mathcal S}-R_{\mathcal C}$. Notice that, in the case $M$ short, we do not use the assumption on the pinching for $\Gamma.$ 

\vspace{12pt}

\begin{center} \noindent{\bf Case  2 ($M$ tall):} $h_M\geq 2h_{\mathcal C}.$
\end{center}

Using Alexandrov reflection with horizontal and vertical hyperplanes, one gets that the part of $M$ above the plane $t=\frac{h_M}{2}$ is a graph as well as the part of $M$ outside the cylinder over $\Omega.$ Then, we have only to understand the topology of $M\cap \Omega\times[0, \frac{h_M}{2}].$ In what follows we will show that there is no point of $M$ in $\Omega\times[h^*, \frac{h_M}{2}].$ Actually we get event more: there is no point of $M$ in $D(R)\times [h^*, \frac{h_M}{2}],$ where $R$ will be fixed later and $\Omega\subset D(R).$ Then, we use the latter to prove that there is no interior point of $M$ in $\Omega\times[0, h^*].$ 

{ The bound $\delta(n,H)$  on the  size of $\Gamma$ will be determined by  a careful choice of the parameter $d$ of the family of immersed annuli ${\mathcal A}_d$ described in  Section \ref{rotaz}.
 Let us explain first how the choice  of $d$ affects the other quantities involved in the proof.} Fix any $d<0$ such that
\begin{equation}\label{restrizione r0}
2r_{int}-r_{ext}<r_{0,d}<r_{min},
\end{equation}
where $r_{0,d}$ is the solution of the equation \eqref{r0d} and $r_{min}$ is the radius found in Lemma \ref{lemmaRR}.  The reasons of the bounds in \eqref{restrizione r0} will be clear in the following.
The choice of $d$ determines the hypersurface $\mathcal A_d$ discussed in Section \ref{rotaz}, together with the radius $r_{1,d}$ {\em i.e.} the solution of equation \eqref{r1d}.  Consequently,  the height $h^*$ defined in \eqref{h*}, $\rho^*_H$, the radius of the {\em spherical cap} of $\mathcal S_H$ of height $h^*$, i.e. the solution of equation \eqref{rH*} depend on the choice of $d$ as well.

We point out that the choice of $d$ is determined by $\Gamma$ through the inequality \eqref{restrizione r0}, in particular if $r_{ext}\rightarrow 0$, which means that $\Gamma$ shrinks to a point, then $d\rightarrow 0$. Moreover, when $d\rightarrow 0,$ then  all  the radii and height $h^*$ tend to zero by \eqref{eq:rh} and Lemma \ref{raggi D}. It follows that if $r_{ext}$ is small enough we have $h^* \ll 2h_{\mathcal C}.$ 
Furthermore,  
 since $\Gamma$ is compact and pinched, then there exists an $\varepsilon>0$ such that $2r_{int}\geq(1+\varepsilon)r_{ext}$. By Lemma \ref{lemmaRR}, we have 
$$
\varepsilon r_{ext}\leq2r_{int}-r_{ext}<r_{0,d}
$$
hence $\frac{\rho_H^*}{r_{ext}}>\varepsilon\frac{\rho_H^*}{r_{0,d}}$. Finally, taking $d$, and hence $\Gamma$, smaller if necessary,  by Lemma \ref{raggi D} one has
\begin{equation}\label{stima1}
\rho_H^*>3r_{ext}.
\end{equation}

{\bf Claim 2A.} There is a point of $p\in M$ on the hyperplane $\{t=h_M-h^*\}$, at distance at least $\rho^*_H$ from the $t$-axis.\\
Using Alexandrov reflection with respect to horizontal hyperplanes, we can prove that the reflections with respect to $\{t=\frac{h_M}{2}\}$ of the points at maximal height of $M$ belong to the closure of $\Omega$. Up to horizontal isometries we can suppose that one of these points belongs to the $t$-axis, that is, it is of the form $(\sigma, h_M)$ where $\sigma$ was defined in Remark \ref{sigma}.
Denote by $M'$ the part of $M$, above the plane $\{t=h_M-h^*\}$. $M'$ is a graph of height $h^*$. Assume, by contradiction that the distance between $\partial M'$ and the $t$-axis is smaller than $\rho_H^*$. Cut $\mathcal S_H$ with a suitable horizontal hyperplane such that $\mathcal S_H'$, the spherical cap above this plane, has height $h^*$. Translate  $\mathcal S_H'$ up until $\mathcal S_H'\cap M=\emptyset$, then move it down. By the maximum principle there is no interior contact point between the spherical cap and $M'$ till  the boundary of $\mathcal S_H'$ reach the height $t=h_M-h^*.$ Then, $M'$ has height less than  $h^*.$ The latter is a contradiction.\\

{\bf Claim 2B.} The compact domain bounded by $M\cap \{t=h_M-h^*\}$ bounds a disk $D(R)$ with $R>\rho^*_H-3r_{ext}$ and center $(\sigma,h_M-h^*)$.\\
Enclose $\Gamma$ with the geodesic sphere $\mathcal C_{ext}$ defined in Lemma \ref{lemmaRR}, reflect the point $p$ found in Claim 2A with respect to any vertical hyperplane tangent to $\mathcal C_{ext}$ in $\hh^n\times\rr$. In this way we have a hyperbolic Lima\c con $\hl$ in $\{t=h_M-h^*\}$ defined by the base point $p$ and the vertical translation of $\mathcal C_{ext}$ at height $h_M-h^*$. Therefore its parameters are $a>\rho^*_H-r_{ext}$ and $c=r_{ext}$. 
  In fact the distance between $p$ and $(\sigma,h_M-h^*)$ is larger than $\rho^*_H,$ where $\sigma$ was defined in Remark \ref{sigma}. Estimate \eqref{stima1} implies that $a>2c>c$, hence, by Lemma \ref{lemma lima}, $\hl$ has two loops. Arguing as in the proof of Lemma \ref{lemmaRR}, we have that the smaller loop is contained in $W$. The claim follows by Lemma \ref{stima loop} and Remark \ref{sigma}.\\

{\bf Claim 2C.} $M \cap \{ D(R) \times  [h^*, h-h^*] \}= \emptyset$.\\
By Claim 2B, $D(R) \subset W$, and by our choice of $h^*\ll h_{\mathcal S}$, the plane $\{t=h_M-h^*\}$ is above the plane $\{t=h_M/2\}$.
By doing Alexandrov reflection  with horizontal planes, the reflection $D^*(R)$ of $D(R)$ with respect to $\{t=\tau\}$, will be contained in $W$, for all $\tau\in [h_M/2-h^*, h_M]$. Therefore, $M \cap \{ D(R) \times  [h^*, h_M-h^*] \}= \emptyset$.\\

{\bf Claim 2D.}  $M\cap\{0\leq t \leq h^*\}$ is outside the cylinder $\Omega \times \{ 0 \leq t \leq h^*\}$.\\
 Denote by $\Sigma$ the embedded part of the annulus ${\mathcal A}_d$ contained in $ (D(r_{1,d})\setminus D(r_{0,d}))\times [0,h^*]$. By construction, the hypersurface $\Sigma$ has two boundary components, denoted  by $C_0$ and $C_{1}$, both  geodesic hyperspheres in the hyperbolic space. In particular $C_0 \subset P$ and $C_{1} \subset \{t=h^*\}$. By inequality \eqref{h*}, we have that the radius of $C_1$ is smaller than $r_{1,d}$. Moreover, up to horizontal translations we can suppose that the center of $C_0$ is $(\sigma,0)$.
By the pinching of $\Gamma$ and Lemma \ref{raggi D}, we can take $\Gamma$ small enough such that $R>r_{1,d}$ holds, where $R$ is the radius found in Claim 2B. 
Therefore, by Claim 2C, we can translate vertically $\Sigma$ such that $\Sigma\subset W$. 
We recall that our choice in  \eqref{restrizione r0}  yields $r_{0,d}<r_{min}\leq r_{int}$ and that the mean curvature of $\Sigma$ is strictly smaller than that of $M$. By the maximum principle it follows that we can translate $\Sigma$ down, until $C_0$ reaches $P$ again, without having an interior contact point between $\Sigma$ and $M$. Moreover by translating $C_0$ in $\Omega$, we can touch every point of $\Gamma,$ while $C_0$ remains in $\Omega$. Furthermore, by Claim 2B, the pinching of $\Gamma$ and Lemma \ref{raggi D}, we can take $|d|$ small enough such that
\begin{equation}\label{eq:R} 
R>r_{ext}+r_{1,d}-r_{0,d}.
\end{equation}
In fact by Claim 2B we have that $R>\rho^*_H-3r_{ext}$ and $\rho^*_H-3r_{ext}>r_{ext}+r_{1,d}-r_{0,d}$ holds because, when $d\rightarrow 0$, the left-hand side tends to zero slower than the right-hand side.

Claim 2C and \eqref{eq:R} imply that when translating $C_0$ inside $\Omega$, the circle $C_1$ remains inside the disk $D^*(R) \subset \{t=h^*\}$, hence the upper boundary of $\Sigma$ will not touch $M.$
Notice also that $\Sigma$ is a vertical graph over the exterior of $D(r_{0,d})$, hence, translating horizontally $C_0$, $\Sigma$ and $M$ cannot meet at an interior point otherwise we would have a contradiction with the maximum principle. Therefore $M \cap ( \Omega\times [0,h^*] ) \neq \emptyset$.\\

This completes the proof of Theorem \ref{theorem-main}.
\end{proof}

\section*{Appendix: list of notations}

To simplify the reading we summarize the principal notations introduced. 

\begin{itemize}
\item[1)] Hypersurfaces of $\hh^n\times\rr$:
\begin{itemize}
\item[$\mathcal S_H$:] compact rotationally symmetric hypersurface with constant mean curvature $H>\frac{n-1}{n}$, see Section \ref{rotaz}; 
\item[$\mathcal S_{\frac{n-1}{n}}$:] complete rotationally symmetric entire graph with constant mean curvature $H=\frac{n-1}{n}$, see Section \ref{rotaz}; 
\item[$\mathcal A_{d}$:] rotationally symmetric annulus with self intersection and constant mean curvature $H=\frac{n-1}{n}$, see Section \ref{rotaz};\\
\item[$\mathcal C_H$:] horizontal cylinder with constant mean curvature $H>\frac{n-1}{n}$, see Section \ref{hor-cyl}.

\end{itemize}

\item[2)] {Heights:} 
\begin{itemize}
\item[$h_{\mathcal S}$:] height of the half sphere $\mathcal S_{H}$, see Section \ref{rotaz};
\item[$h^*$:] approximated value of the height of the portion of $\mathcal A_d$ between $r_{0,d}$ and $r_{1,d}$, see \eqref{h*};
\item[$h_{\mathcal C}$:] height of half of the cylinder $\mathcal C_H$, see Section \ref{hor-cyl};
\item[$h_M$:] height of $M$.
\end{itemize}

\item[3)] Radii:
\begin{itemize}
\item[$R_{\mathcal S}$:] radius of $\mathcal S_{H}$, it is the unique solution of \eqref{rs}; 
\item[$r_{0,d}$:] the minimum radius for which $\mathcal A_d$ is defined, it is the unique solution of \eqref{r0d};
\item[$r_{1,d}$:] the radius for which $\mathcal A_d$ has horizontal tangent plane, it is the unique solution of \eqref{r1d};
\item[$\rho^*_H$:] radius of the spherical cap of $\mathcal S_{H}$ with height $h^*$, it is the unique solution of equation \eqref{rH*};
\item[$R_{\mathcal C}$:] radius of $\mathcal C_{H}$, it is the unique solution of \eqref{rc};
\item[$r_{int}$:] interior radius of $\Gamma$;
\item[$r_{ext}:$] exterior radius of $\Gamma$;
\item[$r_{min}:$] the minimum radius on which $M$ is a graph, it is determined by Lemma \ref{lemmaRR}.
\end{itemize}

\end{itemize}


\end{document}